\documentclass{article}
\usepackage{amsmath,amssymb}
\usepackage{graphicx,color}

\newtheorem{theorem}{Theorem}[section]

\begin{document}

\title{Approximation of solutions of parameter-dependent problems by residual 
neural networks}
\author{Ana Carpio,   \\ 
Universidad Complutense de Madrid, Spain}

\maketitle

{\bf Abstract.}
We develop a convergent scheme to train neural networks involving
analytic activation functions based on gradient flows. Convergence
properties are guaranteed by Lojasiewicz theory. The main advantage of this
approach is its simplicity of implementation. The coefficients of the network are
approximated by solving a system of ordinary differential equations. We test
the method by constructing residual neural network approximations of solutions of
parametric problems. The dependence of the solutions of simple ordinary differential 
equations on a few parameters is correctly reproduced. The solutions of inverse
problems involving wave constraints which depend on a few parameters can be
reasonably approximated, even in regions in which the problem is severely ill
posed. \\


{\bf MSCcodes.}
68Q32, 65D15, 65K10.   \\

\section{Introduction}
\label{sec:intro}

A wide variety of problems in applied sciences require solving
large amounts of parameter-dependent ordinary or partial
differential equations.  For instance, technological designs 
often rely on large scale simulations
\cite{springer_texts,largescale,chapman_texts} to test the 
response to variations of relevant parameters.
Likewise, optimization and Bayesian approaches 
to inverse problems typically involve solving differential 
constraints for huge amounts of parameter proposals 
\cite{sergei,omar,sullivan}.
Sometimes, the full solution is not really needed for these
applications, quantities of interest calculated from them
are used in practice instead. 

Over the years, a broad range of numerical schemes have been
developed to approximate the solutions of models formulated 
in terms of differential equations with a controlled error
\cite{leveque,quarteroni}. The increase in computer speed, 
combined with high performance  computing tools, allows us to 
handle a growing variety of  complex problems employing these 
methods for given parameter choices
\cite{springer_texts,largescale,chapman_texts}.
Whenever we vary a parameter, we must face the computational
cost of solving the problem again. In that respect, the development 
of approximation procedures which are able to provide reasonable 
solutions for multiple parameter choices at low cost would constitute 
an important breakthrough. 
Machine learning techniques, in particular neural network (NN) 
based approximations \cite{book},  have raised some 
expectations in that respect.  
We will consider here their potential for the approximation of solutions
of multi-parametric problems and inverse problems.

Universal approximation theorems establish the density of
several NN architectures within given spaces of interest
\cite{cybenko, gripenberg, hornik, kidger, continuous, lu, 
maiorov, park, pinkus, tabuada}. Quantitative refinements
discuss the number of layers/units to achieve a given error
for different architectures, choices of activation function
and target regularity
\cite{heinecke, johnson, kratsios, lin, lu2, poggio, shen}.
Most results refer to continuous or integrable functions 
between finite dimensional Euclidean spaces with either
the topology of uniform convergence  or $L^p$ norms.
Given an element of the space, they state that it is possible 
to find weights for the selected type of NN architecture in 
such a way that the error in the selected norm is as small 
as wished provided the number of layers or nodes is large
enough, or specific activation functions are selected.
However, they rarely provide constructions of such 
approximants. Failure in applications is attributable
to a deficient  selection of weights,  number of hidden units
and number of layers, or to the appearance of a stochastic 
rather than a deterministic  connection between input and 
target. 

Once a network architecture with a definite number
of layers and units per layer is chosen, the free parameters
are usually found by minimizing the error
between the network predictions and the available data
for a set of inputs \cite{book}. Success of standard
approaches, such as gradient descent or Newton methods,
typically requires that the cost has isolated critical points,
some convexity assumptions or nondegeneracy of the
hessian. In this paper, we explore a different strategy
which does not require such conditions and works for
analytic neural network approximations. We use the cost to 
define a gradient flow which converges to a well defined limit 
point \cite{absil, harauxloja, gradientloja}
and provides the sough values for the network weights.
We test the procedure on residual neural networks for function
approximation. Given a continuous function between
finite dimensional Euclidean spaces, universal approximation 
theorems establish the existence of deep enough residual 
neural networks for which the error is as small as needed 
\cite{continuous, lin, tabuada}. 
The solutions of differential equations depend continuously on 
model parameters in many situations. We will illustrate the
potential of the method approximating simple parametric
solutions and inverting operators defined by means of differential
equations

For ease of the reader, Section \ref{sec:approximation} recalls
the approximation set-up and fixes terminology and notations.
Section \ref{sec:analytic} discusses the practical construction of 
approximants from sample data when the network is analytic. 
Instead of standard optimization procedures, we resort to gradient 
flows to calculate the network free coefficients. Section \ref{sec:results}
tests the procedure on parameter-dependent examples: a simple 
ordinary differential equation and an inversion problem defined by
a wave equation.
Section \ref{sec:latin} discusses the possibility of reducing the
amount of data by replacing uniform grids in the parameter space
by latin hypercube samples.
Finally, section \ref{sec:conclusions} contains our conclusions.

\section{Approximation of functions by neural networks}
\label{sec:approximation}

Consider a set of inputs ${\cal X} \subset \mathbb R^N$
and a set of outputs  ${\cal Y} \subset \mathbb R^M$. We wish
to approximate some hidden relation represented by an unknown
target function $F: {\cal X}  \rightarrow {\cal Y}$.
The neural network (NN) approach   constructs 
approximants $\hat F$ by a sequence of transformations. 
Given an input 
$\mathbf x \in \mathbb R^N$ we build successive layers for $\ell=1, 
\ldots, L$  by
\begin{itemize}
\item making $n_\ell$ linear combinations
\[ \tilde x_j^{(\ell)} = \sum_{i=1}^{n_{\ell-1}} w_{j,i}^{(\ell)} x_i^{(\ell-1)} 
+ b_j^{(\ell)}, \quad j=1, \ldots, n_\ell, \]
\item applying an activation function $\sigma_\ell$ 
to obtain
\[ x_j^{(\ell)} = \sigma_\ell(\tilde x_j^{(\ell)}), \quad j=1, \ldots, n_\ell,\]
\end{itemize}
where $w_{j,i}^{(\ell)} $ and $b_j^{(\ell)}$ are weights and biases to
be selected. We set  $n_0=N$, $\mathbf x^{(0)} =\mathbf x$,
$n_L=M$, and $\mathbf x^{(L)} =\mathbf y$.
In this way, we define a set of approximant functions ${\cal H}$ 
\begin{eqnarray*}
{\cal H} = \{ F_{\boldsymbol \gamma}: {\cal X} \rightarrow {\cal Y},
\quad \boldsymbol \gamma \in \Gamma\}
\end{eqnarray*}
depending on parameter sets $\boldsymbol \gamma$ formed
by all the involved weights and biases. 
The intermediate layers $\mathbf x^{(\ell)}$ for $\ell=1,\ldots,L-1$ are  
called hidden units, whereas $\mathbf x^{(L)}$ is the output layer. 
Usually, the activation functions $\sigma_\ell = \sigma$ for $\ell=1,
\ldots, L-1$, whereas $\sigma_L=g$ can differ. Depending on the 
application,  linear functions, sigmoids, hyperbolic tangents, softplus, 
softmax, rectifier linear units, and so on, are possible choices 
\cite{book}. Variations in this structure give raise to particular NN 
families. We focus here on a remarkably successful architecture: 
deep residual neural networks (ResNets). These networks are built
by means of the recurrence
\begin{eqnarray} \begin{array}{l}
\mathbf x^{(\ell)} = \mathbf x^{(\ell-1)} + \mathbf V^{(\ell)}
\sigma(\mathbf W^{(\ell)} \mathbf x^{(\ell-1)} + \mathbf b^{(\ell)}),
\quad \ell = 1, \ldots, L-1, \\
\mathbf x^{(L)} = g(\mathbf x^{(L-1)} ),
\end{array} \label{resnet}\end{eqnarray}
starting from input values $\mathbf x^{(0)}=\mathbf x$ and applying 
$\sigma$  and $g$ elementwise. For fully connected layers
\[
\mathbf W^{(\ell)} \in {\cal M}_{q \times n}(\mathbb R),  \;
\mathbf V^{(\ell)} \in {\cal M}_{n \times q}(\mathbb R),  \;
\mathbf b^{(\ell)} \in \mathbb R^q.
\]
Universal approximation results for ResNets
are available under different conditions on the terminal
functions $g: \mathbb R^N \rightarrow \mathbb R^M$ and the
intermediate control functions $f_{\boldsymbol \gamma_\ell}: 
\mathbb R^N \rightarrow \mathbb R^N$, given here by
\[ f_{\boldsymbol \gamma_\ell}(\mathbf z) = \mathbf V^{(\ell)}
(\mathbf W^{(\ell)} \mathbf z + \mathbf b^{(\ell)}),\quad
\boldsymbol \gamma_\ell = (\mathbf V^{(\ell)}, \mathbf W^{(\ell)}, 
\mathbf b^{(\ell)}).\]
%
%
If  $N \geq 2$, $q \leq N$ and $K \subset \mathbb R^N$ is compact,
we can approximate continuous functions 
$F: K  \rightarrow \mathbb R^M$ in $L^p$ norm, $1 \leq p <\infty$,
provided the control family ${\cal F}$ and the
terminal family ${\cal G}$ satisfy
\begin{itemize}
\item[1)] For any compact $K \subset \mathbb R^N$, there exists
a Lipschitz function $g \in {\cal G}$ such that $F(K) \subset
g(\mathbb R^N)$.
\item[2)] The closure $\overline{CH}({\cal F})$ of the convex hull of 
${\cal F}$ contains a well function $h$, that is, a function whose 
components $h_i:  \mathbb R^N \rightarrow \mathbb R$, $i=1,
\ldots, N$ are Lipchitz functions satisfying
$\{\mathbf x \in \mathbb R^N \, | \, h_i(\mathbf x) = 0 \}=Z_i$ is the
closure of an open convex set $\Omega_i \subset \mathbb R^N$,
$ Z_i = \overline{\Omega_i}$.
\item[3)] ${\cal F}$ is restricted affine invariant, that is, for any
$f \in {\cal F}$, $Df(A \cdot + b) \in {\cal F}$ for $b \in \mathbb R^N$,
$A$, $D$ diagonal matrices $N \times N$,  such that
the entries of $D$ are $\pm 1$, $0$ and the entries of $A$ are
smaller than or equal to $1$ in absolute value.
\end{itemize}
See \cite{continuous} for details. Condition 1) is satisfied by any 
Lipschitz and surjective function $g$.  
Conditions 2) and 3) are fulfilled by activation functions $\sigma$
commonly  used in the design of neural networks, such as sigmoids 
or ReLU functions. When the activation function $\sigma$, or one
of its derivatives, satisfies a quadratic differential equation, any 
continuous function $F: \mathbb R^N \rightarrow \mathbb R^N$
can be approximated arbitrarily well, on a compact set and with 
respect to the supremum norm, by deep enough ResNets with
$N+1$ units per layer \cite{tabuada}. Many popular activation functions 
satisfy this condition:
\begin{itemize}
\item the logistic sigmoid $\sigma(z) = 1/(1 + e^{-z})$, 
\item the hyperbolic tangent $\sigma(z) = \tanh(z)$, 
\item the soft plus function $\sigma(z) = \log(1+e^{rz})/r$, $r>0$
(a smooth approximation to the ReLU function $\sigma(z) = {\rm max}(z,0)$ 
as $r \rightarrow \infty$), 
\item the leaky soft plus function $\sigma(z) = rx+ \log(1+e^{(1-r)kz})/k$,
 $r,k>0$ (a smooth approximation to the leaky ReLU function 
 $\sigma(z) = z$ if $z\geq 0$ and $\sigma(z) = rz$, if $z<0$, 
as $k \rightarrow \infty$).
\end{itemize}
These results, however, do not provide algorithms to construct 
the desired approximants.

\section{Construction of approximants}
\label{sec:analytic}

In practice, given a target function $F$, we can construct
neural network approximants by solving an optimization 
problem. Let ${\cal H}$ denote the set of approximants 
${\cal H} = \{ F_{\boldsymbol \gamma}: {\cal X} \rightarrow {\cal Y},
\; \boldsymbol \gamma \in \Gamma\}$
depending on parameter sets $\boldsymbol \gamma$. 
When we have data samples 
$F(\mathbf x_j) = \mathbf y_j$, $j=1,\ldots, J$, we solve
\begin{eqnarray}
{\rm min}_{\boldsymbol \gamma \in \Gamma} \; E(\boldsymbol \gamma)
= {\rm min}_{\boldsymbol \gamma \in \Gamma} \; {1\over J} 
\sum_{j=1}^J \|F(\mathbf x_j) - F_{\boldsymbol \gamma}(\mathbf x_j)\|^2.
\label{energy}
\end{eqnarray}
Depending on the application, one can replace the square loss by a 
different loss function \cite{book}. We are interested in a global 
minimum. When several local minima appear, we choose the one
with smaller cost. Standard approaches implement variants of gradient 
descent or Newton type techniques \cite{book,optimization}. The calculation 
of derivatives and hessians exploits backpropagation techniques \cite{lecun}.
Success of these approaches typically requires that the cost has isolated 
critical points, some convexity assumptions or nondegeneracy of the
hessian, which may force the addition of regularizing terms.

We propose here a different strategy, specially useful when the neural 
network architecture, more precisely, the activation functions, are analytic.
This is the case of sigmoidal functions, hyperbolic tangents and soft
plus activations, for instance.
Let $\boldsymbol \gamma$ be the vector formed by all the parameters 
to be identified by minimizing the cost $E(\boldsymbol \gamma)$.  
Local minima are characterized as solutions of 
$ \nabla E(\boldsymbol \gamma) = 0$. Instead of resorting to optimization
techniques, we take a dynamical systems viewpoint and consider a gradient 
system with an artificial time $t$ and  let the system relax to an equilibrium, 
the so called relaxation approach. When the gradient system is analytic,
convergence is guaranteed thanks to Lojasiewicz type inequalities.

\subsection{Lojasiewicz inequalities}
\label{sec:loja}

Consider the gradient system 
\begin{eqnarray}
\boldsymbol \gamma_t + \nabla E(\boldsymbol \gamma) = 0
\label{gradient}
\end{eqnarray}
where $E \in C^2(\mathbb R^P,\mathbb R)$.
If $\boldsymbol \gamma \in C^2([0,\infty),\mathbb R^P)$ is a 
solution, then
\[
{d \over dt}E(\boldsymbol \gamma) = 
< \nabla E(\boldsymbol \gamma), \boldsymbol \gamma' >
 = - \| \boldsymbol \gamma_t \|_2^2 \leq 0,
\]
so that $E(\boldsymbol \gamma)$ is a Liapunov function 
\cite{harauxbook} for the gradient system. If $\boldsymbol \gamma$ 
is bounded, then $\| \boldsymbol \gamma_t \|_2 \rightarrow 0$ as 
$t \rightarrow \infty$ and 
\[
{\rm lim}_{t \rightarrow \infty} {\rm dist}(\boldsymbol \gamma(t), 
{\cal E}) = 0,
\]
where we denote the set the equilibrium points by 
${\cal E} = \{ \mathbf a \in \mathbb R^P \, | \, \nabla E(\mathbf a) = 0 \}$.
This set may have an arbitrary structure. It may be  formed by a 
single point, a collection of isolated points, a continuum. In general,
Convergence to the set of equilibrium points does not imply
convergence to a particular equilibrium point. A remarkable
exception occurs when $E$ is analytic:
when $P \geq 2$, we can find $\mathbf a \in {\cal E}$ such that
\[
{\rm lim}_{t\rightarrow \infty} \boldsymbol \gamma(t) = \mathbf a,
\]
provided $E$ is analytic \cite{harauxloja, loja3}. This result is not
necessarily true even if $E \in C^{\infty}$. 
It relies on  Lojasiewicz inequality \cite{loja1,loja2}.
\vskip 1mm

\begin{theorem} (Lojasiewicz inequality)
Let $U \subset \mathbb R^P$ an open set and
$E: U \rightarrow \mathbb R$ an analytic function.
For any $\mathbf a \in {\cal E}$ there exists a neighbourhood 
$\omega \subset U$ of $a$ and a real number 
$\alpha \in (1/2, 1)$ such that
\begin{eqnarray*}
| E(\mathbf u) - E(\mathbf a)|^\alpha \leq ||\nabla E(\mathbf u)||_2, \;
\forall \mathbf u \in \omega.  
\end{eqnarray*}
\end{theorem}
This inequality has some immediate consequences:
\begin{itemize}
\item Given $\Gamma \subset {\cal E}$ compact and connected,
we have
\begin{itemize}
\item $E(\mathbf  a)=E(\mathbf  b)=c$, $\forall \mathbf  a, 
\mathbf  b \in \Gamma$,
\item There is a neighbourhood $\omega \subset U$ 
of $\Gamma$  and $\alpha \in (1/2, 1)$ such that
\[
| E(\mathbf  u) - c|^\alpha \leq ||\nabla E(u)||_2, \;
\forall \mathbf  u \in \omega.
\]
\end{itemize}
\item If satisfied for a value $\alpha$, it also holds for 
$\alpha' \in (\alpha,1)$. We would like to take $\alpha$
as small as possible. For the optimal $\alpha$, and additional
constant $M>0$ appears, that is,
\begin{eqnarray}
| E(\mathbf u) - c|^\alpha \leq M ||\nabla E(\mathbf u)||_2, \;
\forall \mathbf u \in \omega. \label{loja}
\end{eqnarray}
\end{itemize}
When a Lojasiewicz inequality holds, the gradient flow
converges to a specific equilibrium \cite{harauxloja, loja3}.
\vskip 1mm

\begin{theorem} (Convergence of the gradient flow) 
Let $U \subset \mathbb R^P$ an open set and 
$E \in C^2(U,\mathbb R)$. Let $\boldsymbol \gamma$ be 
a solution of (\ref{gradient}) such that $\boldsymbol \gamma(t) \in K$ 
for all $ t \geq 0$ for  some compact set $K \subset U$. Then,
\begin{eqnarray*}
\Gamma = \int_{\tau>0} \overline{\cup_{t \geq \tau} 
{\boldsymbol \gamma(t)}} \subset {\cal E}.
\end{eqnarray*}
Moreover, if (\ref{loja}) holds in a neighbourhood $\omega
\subset U$ of $\Gamma$ for some $c \in \mathbb R$
and $\alpha \in (1/2, 1)$, then there is $a \in {\cal E} \cap
E^{-1}(c)$ such that
\begin{eqnarray*}
\| \boldsymbol \gamma(t) - \mathbf a  \|_2 \leq C \, t^{-\eta(\alpha)}, 
\quad
\eta(\alpha) = {1- \alpha \over 2 \alpha -1}, \, \forall t>0, \\
| E(\boldsymbol \gamma(t)) - c| \leq C \, t^{-\beta(\alpha)}, \quad  
\beta(\alpha) = {1\over 2 \alpha -1}, \, \forall t>0, 
\end{eqnarray*}
for some $C>0$. If (\ref{loja}) holds with $\alpha = {1\over 2}$
then, for some $\delta >0$, 
\begin{eqnarray*}
\| \boldsymbol \gamma(t) - \mathbf a  \|_2 \leq C \, e^{-\delta}, 
\quad \forall t>0.
\end{eqnarray*} 
\end{theorem}
See \cite{harauxloja, loja3} for details.
The number $1-\alpha$ is called the Lojasewicz exponent
in the literature. Large values of $\alpha$ result in slow
convergence. The optimal $\alpha$ is usually estimated
from the local behavior near the limit instead of being a
global exponent.

These results are the basis of relaxation methods often used to
calculate stationary configurations of complicated energies
\cite{prb}. They have inspired a number of convergence results
for descent methods in a variety of optimization
frameworks \cite{absil,gradientloja}.

\subsection{Application to neural network fitting}
\label{sec:application}

In the context of neural network based approximation,
we have the following result.
\vskip 1mm

\begin{theorem} (Network fitting)
Consider a set ${\cal H}$ 
formed  by neural  networks $F_{\boldsymbol \gamma}$ 
governed by parameters 
${\boldsymbol \gamma} \in \Gamma \subset \mathbb R^P$
and constructed selecting analytic activation functions $\sigma$.
Then, the associated gradient flow defined by (\ref{gradient})
converges to a point $\mathbf a \in \mathbb R^P$ such that
$\nabla E(\mathbf a) =0$. Moreover, 
$\| \boldsymbol \gamma(t) - \mathbf a  \|_2$ decays
as $t^{-\eta}$ for some $\eta>0$. This point $\mathbf a$ 
provides the optimal weight and biases for the network.
\end{theorem}

{\bf Proof.} It is an immediate consequence of Lojasewicz
results for analytic gradient flows stated in Section 3.1.\\

Typical choices for activation functions, such as the
logistic sigmoid, the hyperbolic tangent, the soft plus and
leaky soft plus functions fulfills these conditions. Therefore,
we can calculate the weights of neural networks constructed
with them by solving gradient systems with standard
solvers for systems of ordinary differential equations. 
Low order adaptive solvers provide efficient tools to this
purpose. 

The decay exponent $\eta$ is specified in the convergence
theorem in terms of the Lojasiewicz exponent $\alpha$.
A criterion to determine $\alpha$, given an analytic function,
is stated in \cite{gwod}. Given an analytic function
$E: U \longrightarrow \mathbb R$ defined in an open set 
$U \subset \mathbb R^P$ and a vector $\mathbf v \in \mathbb R^P$,
we define the polar curve
\begin{eqnarray*}
\Gamma_{\mathbf v} = \{  \mathbf u \in U \, : \, \exists \lambda \in \mathbb R
\big| \nabla E(\mathbf u) = \lambda \mathbf v \} = \nabla E^{-1} 
(\mathbb R \mathbf v).
\end{eqnarray*}
In particular, for the standard canonical basis $\mathbf e_1, \ldots,
\mathbf e_n$, we denote by $\Gamma_{\mathbf e_1}, \ldots, 
\Gamma_{\mathbf e_n}$ the corresponding polar curves.
If $U$ is a neighborhood of $\mathbf 0$, we define $\nu(f)$ as
the largest integer $k$ such that $t^{-k}E(t)$ is bounded
near zero.  
\vskip 1mm

\begin{theorem} (Exponent value)
Consider an analytic function
$E: U \longrightarrow \mathbb R$ defined in a neighborhood  $U$
of $0 \in \mathbb R^P$. Assume that $E^{-1}(\mathbf 0)= \{\mathbf 0\}$.
The Lojasiewicz exponent $\alpha$ for $E$ is
$\alpha = {\nu(\nabla E \circ \gamma_j) \over \nu(E \circ \gamma_j)}$
for some $\gamma_j: (-1,1) \rightarrow \Gamma_{\mathbf e_j}$, 
$\gamma(0)=0$, $j=1,\ldots,P$.
\end{theorem}

See \cite{gwod} for details.
This Theorem, however, does not specify which $\gamma_j$ works.

\begin{figure}[!htb] \centering
\includegraphics[width=10cm]{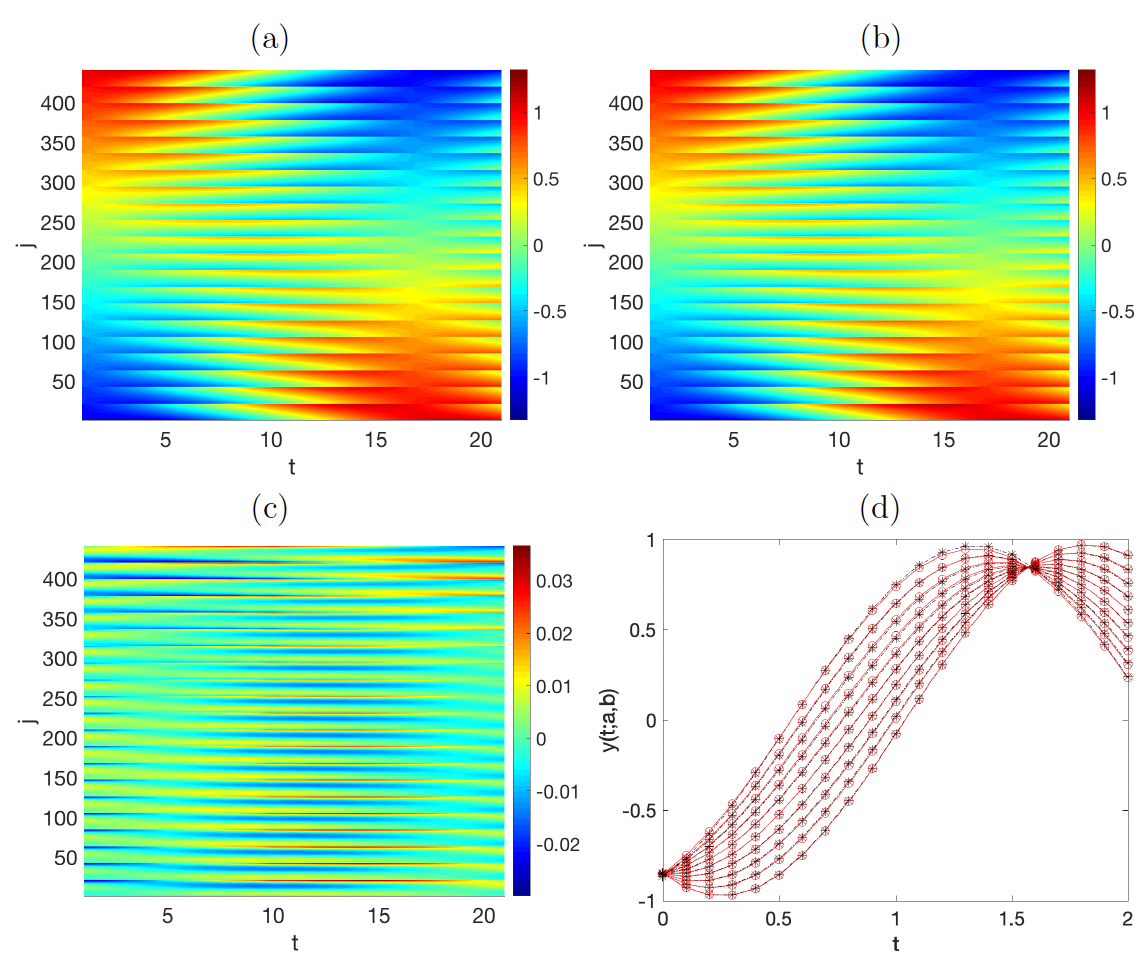}
\caption{(a) exact solutions $y(t,p_j)$, $j=1,\ldots,441$ versus (b) approximations
obtained with the trained Resnet and (c) error. Panel (d) compares the Resnet approximant (blue asterisks) and the exact solutions $y(t;a_r',b_s')$, $a_r'=-0.85$, $b_s'=-0.95+ (s-1)0.2$, $s=1,\ldots,10$ (red circles) at points outside the training grid.}
\label{fig1}
\end{figure}


\section{Numerical tests}
\label{sec:results}

Let us illustrate the method with examples in function
approximation involving differential equations.

\subsection{Parametric solutions of ordinary 
differential equations}
\label{sec:forward}

Consider the second order differential equation
\begin{eqnarray*}
y''(t) + k y(t) = 0, \quad t \in [0,2] \\
y(0) = a, \; y'(0) = b, 
\end{eqnarray*}
for $(a,b) \in [-1,1]$. We wish to approximate the
solution $y(t;a,b)$, which depends continuously on
$a$ and $b$. We  fix $k=4$ and define a mesh
$t_\ell=(\ell-1) \, \delta t$, $\delta t =0.1$, $\ell=1,\ldots,21$. 
Then, we define the function
\begin{eqnarray*}
\begin{array}{llll}
F: & \mathbb R^2 & \rightarrow & \mathbb R^{21} \\
    & (a,b)  &  \rightarrow & (y(t_\ell;a,b))_{\ell=1}^{21}
\end{array}
\end{eqnarray*}
and seek a residual neural network approximant. To
construct it, we use data for a grid $(a_r,b_s)$ with step 
$da=db=0.1$, $a_r=-1+ (r-1) \, da$, $b_s=-1+(s-1) \, db$, 
$r,s=1, \ldots, 21$, that is, we assume we know the values
$y(t_\ell;a_r,b_s)$ for $r,s,\ell$. In fact, there are
well known explicit expressions for the solutions of this
problem
\begin{eqnarray*}
y(t;a,b) =  a cos(2t) + b/2 sin(2t), \quad t>0.
\end{eqnarray*}
Relabelling $x_{rs}=(a_r,b_s)$ with a 
single index $j=S (r-1) + s$, we get the data 
$F(x_j) =  y_j$, $j=1,\ldots,J$,
$J=441$, $S=21$. We seek a ResNet given by (\ref{resnet})
with $N=2$, $M=21$, using for instance $L=1$ layer
and $q=5$ units, with sigmoidal activation function
$\sigma$. We characterize $g: \mathbb R^N \rightarrow
\mathbb R^M$ by a matrix 
$\mathbf G \in {\cal M}_{MxN}(\mathbb R)$. The parameters
$\boldsymbol \gamma$ to be identified are the coefficients
of $\mathbf G$, $\mathbf W^{(1)}$, $\mathbf V^{(1)}$ and 
$\mathbf b^{(1)}$. We consider the cost (\ref{energy})
and solve the gradient system (\ref{gradient}). For
the starting point, we generate a collection or random values
and select the proposal with smallest cost. The gradient
system converges quite fast to an equilibrium, providing
all the coefficients of the neural network.
The system was solved with MATLAB routine ode23 and the 
cost consistently decreased. We stopped the evolution when
variations in the cost fell below a tolerance. It took about 5 
minutes on a common laptop. Figure \ref{fig1} illustrates
the performance of the calculated approximant on the
training mesh and outside it, but remaining in the compact 
set $[-1,1]^2$.
%

\subsection{Inversion problems constrained by partial differential 
equations}
\label{sec:inverse}

Consider a wave equation
\begin{eqnarray*} \begin{array}{ll}
u_{tt} - c^2 u_{xx} = 0, & x \in \mathbb R, t > 0,      \\[1ex]
u(0,x) = h \exp(-10 w \, (x-1)), & x \in \mathbb R,   \\[1ex]
u_t(0,x) = 0, & x \in \mathbb R, 
\end{array} \label{wave} \end{eqnarray*}
with parameters $h \in [1,2]$, $w \in [1,2] $, $c \in [0.1, 2]$.
The forward problem consists in finding the solution $u(x,t)$ of this equation 
given parameters $(h,w,c)$. It has a well known explicit expression
\begin{eqnarray*}
u(t,x;h,w,c) = {u(0,x+ct;h,w,c) + u(0,x-ct;h,w,c) \over 2}, \quad
x \in \mathbb R, t > 0.
\end{eqnarray*}
The inverse problem consists in finding the parameters $(h,w,c)$ given 
the solution $u$. In practice,
one considers finite dimensional approximations of $u$, or just values
$d_{r,\ell}$, $r=1,\ldots,R$, $\ell=1,\ldots,{\cal L}$, at a finite collection of 
points. We wish to approximate a function $f: \mathbb R^{R\times S}
\rightarrow \mathbb R^3$.  Without considering the well posedness of
this type of inverse problems, we just see if we can construct a residual 
neural network approximant.

We introduce a mesh
$t_\ell=(\ell-1) \, \delta t$, $\delta t =0.1$, $\ell=1,\ldots,21$,
and consider the values at $x_0=0$, $x_2=2$.
Then, we define the function
\begin{eqnarray*}
\begin{array}{llll}
F: & \mathbb R^{42} & \rightarrow & \mathbb R^{3} \\
    & \mathbf u= (u(x_0,t_\ell), u(x_1,t_\ell))_{\ell=1}^{\cal L}  &  
    \rightarrow & (h(\mathbf u),w(\mathbf u),c(\mathbf u))
\end{array}
\end{eqnarray*}
and seek a RNN approximant.

\begin{figure}[!htb] \centering
\includegraphics[width=10cm]{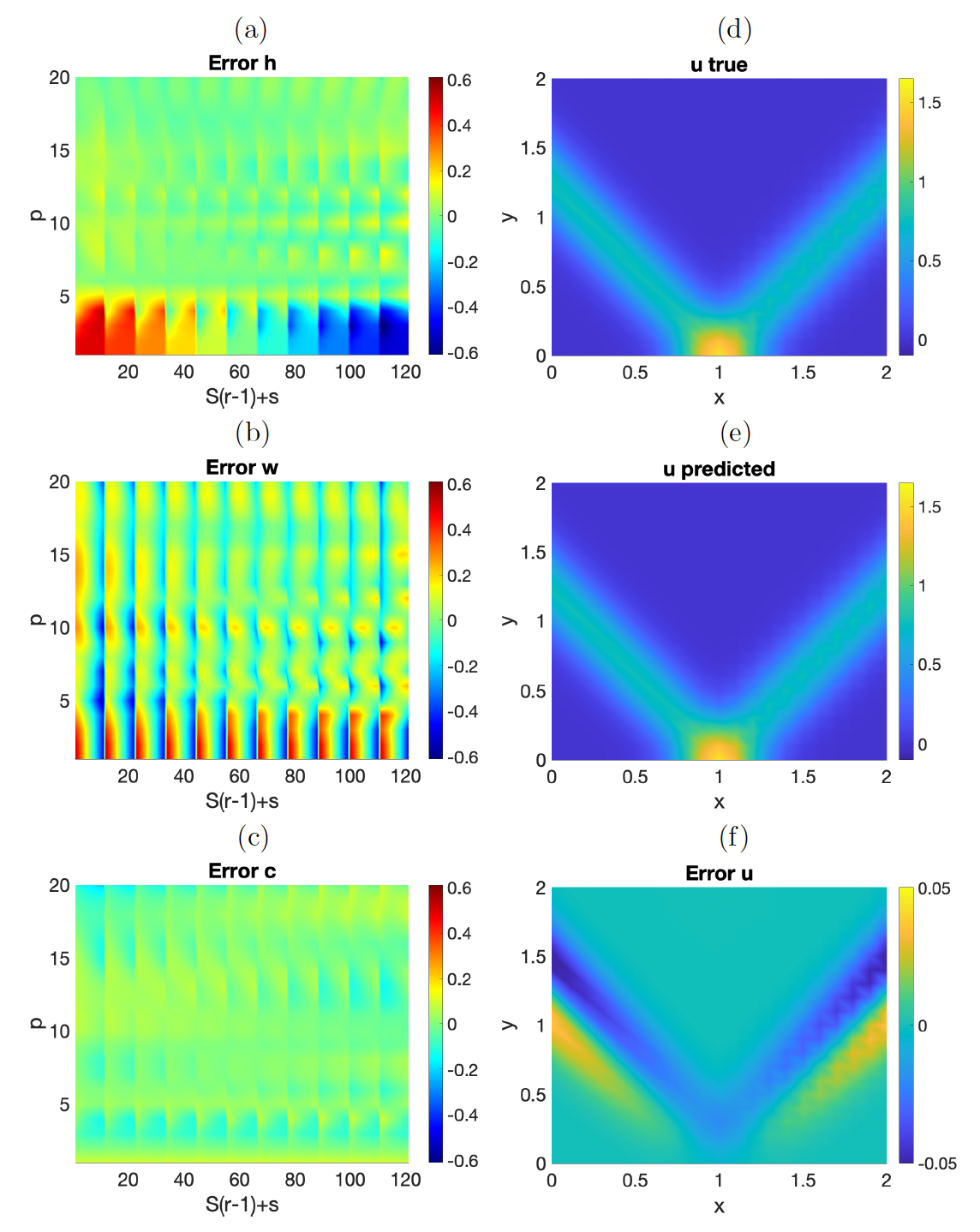}
\caption{
(a)-(c) Error in the true and predicted values of $(h,w,c)$ in the grid $c_p=0.1+
0.1(p-1)$, $p=1,\ldots,20$,  $h_r=0.1(r-1)$, $w_s=0.1(s-1)$, $r,s=1,\ldots,10$.
While velocities $c$ are uniformly well approximated, the approximation of the height 
$h$ and the width $w$ looses quality for small $c$.
(d)-(e) True solution for $h=w=1.55$ and $c=0.8$ versus predicted solution using
the values of $h=1.54,w=1.59,c=0.81$ predicted by the fitted RNN for the 
corresponding data at $0$ and $2$, and absolute error. }
\label{fig2}
\end{figure}

\begin{figure}[!htb] \centering
\includegraphics[width=10cm]{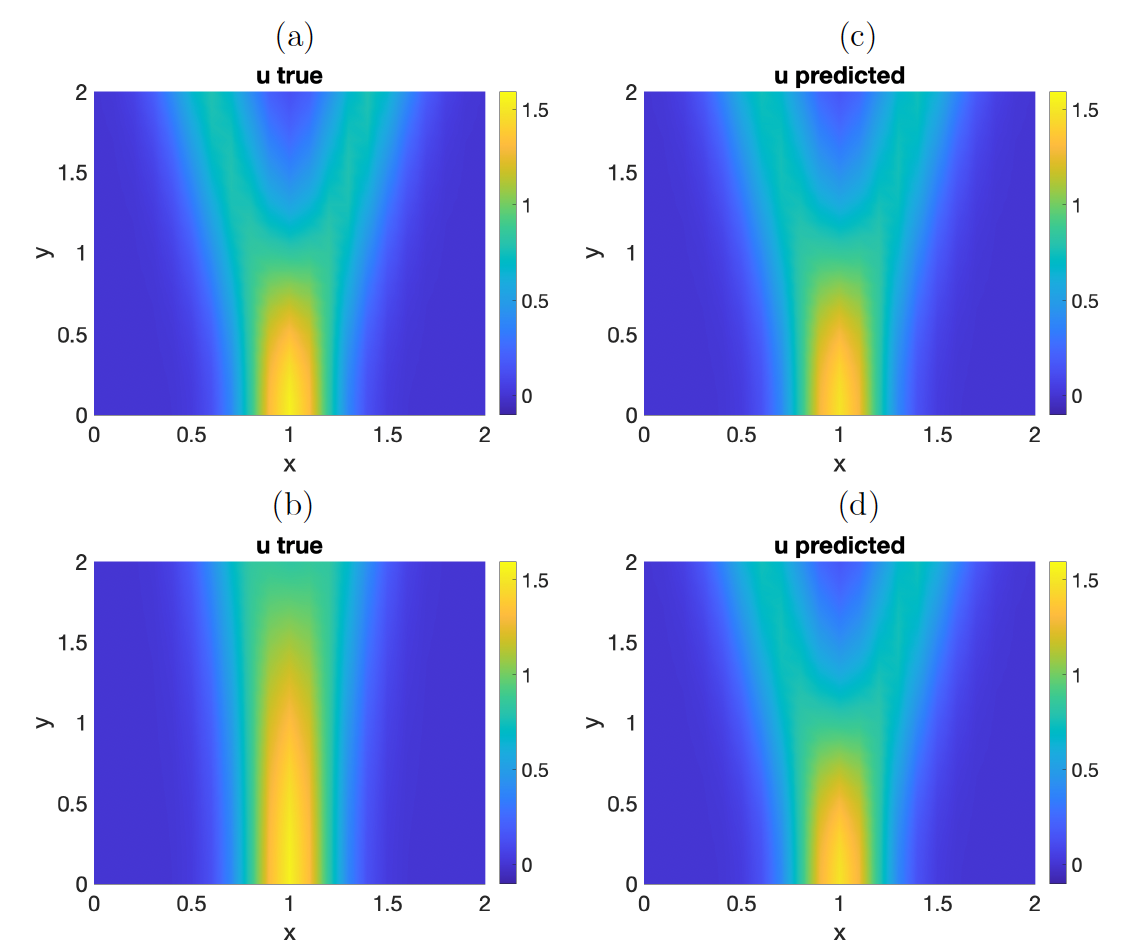} 
\caption{True wave patterns for (a) $c=0.2$,$h=w=1.55$ and (b) $c=0.1$,$h=w=1.55$.
Predicted wave patterns for the predicted (c) $h=1.5381,w=1.4918,c=0.1935$ and
(d) $h=1.394,w=1.4926,c=0.1924$.}
\label{fig3}
\end{figure}

To train the network we generate data on a grid of step $0.1$ in 
$[1,2]^2 \times [0.1,2]$. We vary $(h,w,c)$ in the grid and calculate 
\begin{eqnarray}
\mathbf u_{r,s,p}=(u(x_0,t_\ell;h_r,w_s,c_p), 
u(x_1,t_\ell,h_r,w_s,c_p))_{\ell=1}^{\cal L}, \quad
x_0=0, x_1=2.
\label{datainv}
\end{eqnarray}
Relabelling  with a single index $j=PS(r-1)+ P(s-1)+ p$, 
we get the  data $F(\mathbf u_j)= \mathbf y_j$, 
$j=1,\ldots,J$, $J=RSP=2420$, $R=S=11$, $P=20$.
We seek a ResNet given by (\ref{resnet})
with $N=42$, $M=3$, using for instance $L=1$ layer
and $q=5$ units, with sigmoidal activation function
$\sigma$. Again, we characterize $g: \mathbb R^N \rightarrow
\mathbb R^M$ by a matrix 
$\mathbf G \in {\cal M}_{MxN}(\mathbb R)$. The parameters
$\boldsymbol \gamma$ to be identified are the coefficients
of $\mathbf G$, $\mathbf W^{(1)}$, $\mathbf V^{(1)}$ and 
$\mathbf b^{(1)}$. We solve the gradient system (\ref{gradient})
for the cost (\ref{energy}). To initialize, we generate a collection 
of random values and select the proposal with smallest cost. 
The gradient system converges to an equilibrium, providing all the 
coefficients of the network.

Figure \ref{fig2} illustrates the performance of the calculated 
approximant.  Predictions worsen for low values of $c$. 
Notice that, for low values of $c$, the wave pattern is different,
compare Figure \ref{fig3}(a)-(b) to Figure \ref{fig2}(d). The
wave does not reach $x=0$ and $x=2$ in the given time
$T=2$.
If we restrict our RNN fitting to the range of velocities $c$ for which 
the wave pattern reaches $x_0=0$ and $x_1=2$ for $t \leq 2$, 
then the wave predictions are reasonable.
However, variations in the signal (\ref{datainv}) collected at $x=0,2$ 
for smaller $c$ are not expected to suffice to distinguish $c,h,w$ in 
the given time interval $[0,T]$.
In spite of that, our RNN approximant produces a prediction, as
seen in panels (c) and (d) of figure \ref{fig3}. The prediction is
noticeably worse for $c=0.1$ because the pattern has a different
shape, but it is still reasonable for $c=0.2$ even if the signal recorded
for that value contain no information.

%

\section{Dataset reduction}
\label{sec:latin}

\begin{figure}[!htb] \centering
\includegraphics[width=10cm]{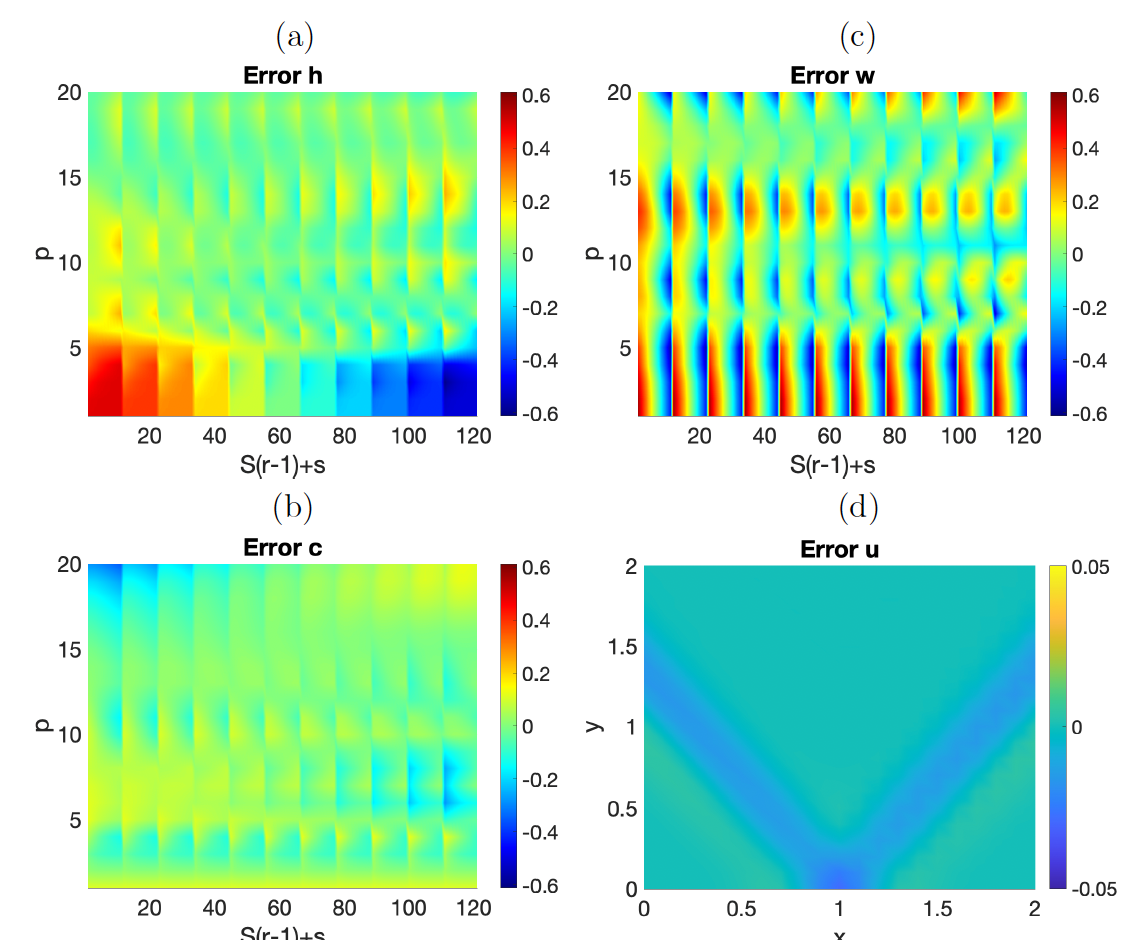}
\caption{
(a)-(c) Error in the true and predicted values of $(h,w,c)$ in the grid $c_p=0.1+
0.1(p-1)$, $p=1,\ldots,20$,  $h_r=0.1(r-1)$, $w_s=0.1(s-1)$, $r,s=1,\ldots,10$.
(d) Absolute error when comparing the true solution for $h=w=1.55$ and $c=0.8$ with
the predicted solution using the values of $h=1.52,w=1.56,c=0.8$ given by the 
RNN fitted with half the number of nodes, chosen by latin hypercube sampling. }
\label{fig4}
\end{figure}

The previous section trains networks using data generated in equispaced
parameter meshes. As the parameter range and the number of parameters 
increase, the amount of data grows fast and training networks becomes 
more costly.
However, the information contained in the data may be redundant. Latin
hypercube sampling (LHS) of the parameter space furnishes an strategy to 
reduce the amount of data while preserving the quality \cite{lhs}.

A square grid is a Latin square if and only if there is only one sample in each 
column and each row. Latin hypercubes generalize this concept to an arbitrary 
number of dimensions: each sample is the only sample in each axis-aligned 
hyperplane containing it. 
One must decide the number of sample points to use and remember in which 
row and column each sample point was taken.
When sampling functions of $N$variables,  we divide the range of each variable 
into $M$ equally probable intervals. Then, we distribute $M$ sample points 
satisfying the Latin hypercube requirements. The number of divisions $M$ 
must be equal for each variable. In this way, we do not require more samples 
for more parameter dimensions.

Figure \ref{fig4} is the counterpart of \ref{fig2} using $1210$ data points
instead of $2420$. A similar cost value is reached three times faster.
For the particular values of $(h,w,c)$  chosen, the approximation of the
solution is even better. However, panels (a)-(c) show that the overall
approximation for general values of $(h,w,c)$ can be worse. Notice that
the cost functions are different since we consider less points. We can 
reduce further the cost values and get approximations comparable to
 those in Figure \ref{fig3}(a)-(c) in a similar number of steps but in two 
 thirds of the time.

\section{Conclusions}
\label{sec:conclusions}
We propose a strategy to train neural networks involving
analytic activation functions based on gradient flows.
Convergence to a solution is ensured by Lojasiewicz theory.
The main advantage is simplicity of implementation.
We  illustrate the procedure training residual neural networks
for simple parametric problems:
an ordinary differential equation and an inversion problem 
defined  by a wave equation, both involving a few parameters.
Residual neural networks are powerful approximants even
for small number of nodes. Simple
codes to solve gradient systems of ordinary differential
equations allow us to fit the coefficients of the networks
yielding reasonable approximants in the tests considered.
Latin hypercube sampling can help to reduce the size.
Interestingly, the neural network trained for the inverse problem
provides predictions of the parameters governing the wave
even for ranges where the data contain basically no information
to invert, which constitutes a remarkable regularizing effect.
In general, the quality of the approximant worsens when
different patterns appear as we vary the parametric regime.
Otherwise it appears to be quite uniform.

\section*{Acknowledgements.} This research was partially supported by the 
FEDER/ MICINN-AEI grants PID2020-112796RB-C21 and  and
PID2024-155528OB-C21.

\end{document}